\documentclass[psamsfonts,a4paper,11pt]{amsart}
\usepackage[all,arc]{xy}
\usepackage{tikz}
\usepackage[all]{xy}
\usepackage{latexsym}
\usepackage{amssymb}
\usepackage{amsmath}
\newtheorem{definition}{Definition}[section]
\newtheorem{lemma}[definition]{Lemma}
\newtheorem{prop}[definition]{Proposition}
\newtheorem{theorem}[definition]{Theorem}

\newtheorem{ques}[definition]{Question}
\newtheorem{remark}[definition]{Remark}
\theoremstyle{definition}
\newtheorem{fact}[definition]{Fact}

\newcommand{\CH}{\mathop{{\rm CH}}\nolimits}

\newcommand*{\lcm}{\mathop{{\rm lcm}}\nolimits}

\newcommand*{\PE}{\mathop{{\rm PE}}\nolimits}

\newcommand{\CB}{\mathop{{\rm CB}}\nolimits}
\newcommand{\Zg}{\mathop{{\rm Zg}}\nolimits}

\newcommand*{\eset}{\varnothing}
\newcommand*{\ex}{\exists}

\newcommand*{\lf}{\lfloor}
\newcommand*{\rf}{\rfloor}

\newcommand*{\ms}{\models}

\newcommand*{\ov}{\overline}
\newcommand*{\seq}{\subseteq}

\newcommand*{\sm}{\setminus}
\newcommand*{\sh}{\sharp}
\newcommand*{\sq}{\sqrt}

\newcommand*{\fty}{\infty}

\newcommand*{\wg}{\wedge}
\newcommand*{\wh}{\widehat}

\newcommand*{\bsm}{\left(\begin{smallmatrix}}
\newcommand*{\esm}{\end{smallmatrix}\right)}
\newcommand*{\bp}{\begin{pmatrix}}
\newcommand*{\ep}{\end{pmatrix}}

\newcommand*{\mN}{\mathcal{N}}

\newcommand*{\A}{\mathbb{A}}
\newcommand*{\C}{\mathbb{C}}

\newcommand*{\N}{\mathbb{N}}

\newcommand*{\Z}{\mathbb{Z}}

\newcommand*{\Ga}{\Gamma}

\newcommand*{\om}{\omega}

\renewcommand*{\phi}{\varphi}

\newcommand*{\Sig}{\Sigma}

\renewcommand{\emph}{\textbf}

\date{}

\begin{document}

\title[]{The Ziegler spectrum of the ring of entire complex valued functions}

\author[]{Sonia L'Innocente}

\address[S.~L'Innocente]{University of Camerino, School of Science and Technologies,
Division of Mathematics, Via Madonna delle Carceri 9, 62032 Camerino, Italy}

\email{sonialinnocente@unicam.it}

\author[]{Fran\c{c}oise Point}

\address[F. Point]{Department of Mathematics (Le Pentagone), University of Mons 20, place du Parc, 
B-7000 Mons, Belgium}

\email{point@math.univ-paris-diderot.fr}

\thanks{The second author is Research Director at the \lq \lq Fonds de la Recherche Scientifique FNRS-FRS".}

\author[]{Gena Puninski}

\address[G.~Puninski]{Belarusian State University, Faculty of Mechanics and Mathematics, av. Nezalezhnosti 4,
Minsk 220030, Belarus}

\email{punins@mail.ru}

\author[]{Carlo Toffalori}

\address[C.~Toffalori]{University of Camerino, School of Science and Technologies,
Division of Mathematics, Via Madonna delle Carceri 9, 62032 Camerino, Italy}

\email{carlo.toffalori@unicam.it}

\thanks{The first and fourth authors were supported by Italian PRIN 2012 and GNSAGA-INdAM}



\begin{abstract}
We will describe the Ziegler spectrum over the ring of entire complex valued functions.
\end{abstract}

\maketitle

\section{Introduction}\label{S-intro}

In \cite{P-T15} the third and fourth authors developed the model theory of modules over B\'ezout domains. For
instance, a substantial information on the structure of the Ziegler spectrum over an arbitrary B\'ezout domain $B$, $\Zg_B$,
was obtained. However, as it was mentioned there, this information is expected to be elaborated for particular
classes of Bezout domains. One example of this refinement was given in \cite{P-T14}, and some information on
the structure of the Ziegler spectrum of the ring of algebraic integers is contained in a recent preprint
\cite{LPT}.

In this note we will investigate this topological space for the prominent example of a B\'ezout domain: the ring
$E= E(\C)$ of complex valued entire functions. This was the question that Luigi Salce once asked Ivo Herzog.
We will show that the points of $\Zg_E$ are given by triples $(U,I,J)$, where $U$ is an ultrafilter on an
(at most countable) nowhere dense subset $D$ of $\C$, and $I, J$ are cuts on the linearly ordered abelian
semigroup $\N^D/U$. The isolated points of this space correspond to principal ultrafilters, hence are of the
form $E_t(k)= E/(z-t)^k E$, where $t\in \C$ and $k\geq 1$, and they form a dense subset in the Ziegler spectrum.

We will also describe the closed points of $\Zg_E$ as the finite length points $E/M^k$ for maximal ideals $M$
of $E$ (for instance the modules $E_t(k)$ are such), plus the generic points. Here generic means the quotient
field of a prime factor $E/P$ of $E$, in particular the quotient field $Q$ of $E$, which is the field of
meromorphic functions.

We will also show that the Cantor--Bendixson derivative $T'_E$ of the theory $T_E$ of $E$-modules coincides
with the theory of $E_S$-modules, where $S$ is the multiplicatively closed set consisting of nonzero polynomials.
There are no isolated points on the next level, i.e. the first $\CB$-derivative $\Zg'_E$ is a perfect space.
Furthermore, no nontrivial interval in the lattice of positive primitive formulae of $T'_E$ is a chain, hence
this theory lacks both breadth and width. Further we will show that the pure injective hull of $E_S$ is a
superdecomposable module $E$-module. Finally we will see that the closed points in $\Zg'_E$ are generics.

This paper paves the way for some future applications, say to the proof of decidability of the theory of
$E$-modules. However we decided to postpone these developments, but spell out now, in a meticulous way, the
facts on the Ziegler spectrum of B\'ezout domains which occur when investigating this space for $E$. We hope
that they will be useful when studying the model theory of modules over other examples of B\'ezout domains which
occur in analysis, say, the ring of real analytic functions.

Due to the fact that none of the authors is an expert in complex analysis we will be quite insisting in collecting
and explaining some facts in this area, which are well known to experts, but were difficult to find for us. To
make up for this we will also include precise references and explanations (mostly taken from \cite{P-T15}) from
model theory of modules over B\'ezout domains.

\section{The ring of entire functions}\label{S-entire}

Let $\C$ denote the field of complex numbers. Recall that a function $f: \C\to \C$ is said to be \emph{entire}, if
it is given by an everywhere convergent power series $\sum_{n=0}^{\fty} a_n z^n$ with complex coefficients $a_n$,
i.e.\ $\lim\limits_{n\to \fty} \sq[n]{|a_n|}= 0$. For instance, the exponential function
$e^z= \sum_{n=0}^{\fty} \frac{z^n}{n!}$ is entire, so as the sine function
$\sin z= \sum_{n=0}^{\fty} \frac{z^{2n+1}}{(2n+1)!}$. More examples and explanations can be found in any complex
analysis textbook, say \cite{Ahl} or \cite{S-Z}. For instance, each entire function is differentiable, and its
derivative is of the same kind.

If we add or multiply entire functions pointwise, the result is likewise. Thus, entire functions form a commutative
ring $E$ whose unity is the constant function of value $1$. We will be interested in ring theoretic properties
of $E$. Note that the cardinality of $E$ is the continuum ${\bf c}= 2^{\aleph_0}$.

Let $Z(f)= \{z\in \C\mid f(z)=0\}$ denote the \emph{zero set} of an entire function $f$. Then $Z(f)$ is at most
countable set whose only possible accumulation point is at infinity. For instance, this is the case for the sine
function: $Z(\sin z)$ consists of points $\pi k$, $k\in \Z$. On the other hand, the zero set of each polynomial
is finite, and the zero set of the exponential function is empty. If $z\in \C$ then $\mu_f(z)$ will denote the
\emph{multiplicity} of $z$ as a root of $f$, which is a natural number, in particular $\mu_f(z)=0$ iff $z$ is not
a zero of $f$. Thus to each entire $f$ we assign the \emph{multiplicity function} $\mu_f: Z(f)\to \N$.
Usually the zeroes of an entire function $f$ are counted as $z_0, z_1, z_2, \dots$ such that
$|z_k|\leq |z_{k+1}|$, and each $z_k$ occurs only finitely many times.

If $f, g\in E$ then clearly $Z(fg)= Z(f)\cup Z(g)$ and, for any $z$, its multiplicity $\mu_{fg}(z)$ is the sum
of multiplicities $\mu_f(z)$ and $\mu_g(z)$. Since the zero set of an entire function is nowhere dense, $E$ is a
\emph{domain}: $fg\neq 0$ for nonzero $f, g\in E$.

The next fact shows that the zero set and the multiplicity of an entire function determine the principal ideal
it generates.

\begin{fact}\label{princ}
Let $f, g\in E$. Then $g\in fE$ if and only if $Z(f)\seq Z(g)$ and $\mu_f(z)\leq \mu_g(z)$ for each $z\in Z(f)$.
In particular $f\in E$ is invertible if and only if $Z(f)= \eset$.
\end{fact}

The proof of this result requires Weierstrass' theorem on functions with a prescribed set of zeroes. Namely, for
each $k$ define the \emph{Weierstrass primary factor} $E_k(z)= (1-z) \exp(z+ z^2/2!+ \dots+ z^k/k!)$, which is
an entire function with $z= 1$ as its only (simple) zero.  Let $\{z_k\}$ be an absolute value nondecreasing
sequence of complex numbers such that each $z_k$ occurs $m_k$ times. Then the infinite product
$\prod_{k=0}^{\fty} E_k(z/z_k)$ is an entire function whose zero set consists of the $z_k$ with multiplicity
$m_k$. Further, if $f$ is any function with this property, then, by \emph{Weierstrass' factorization theorem},
$f= e^g\cdot \prod_{k=0}^{\fty} E_k(z/z_k)$ for an entire function $g$.

A useful variant of this result is the following.

\begin{fact}\label{wnk}(see \cite[Prop. 1.1]{G-H})
Let $\{z_k\}$ be an absolute value nondecreasing sequence of complex numbers with no finite accumulation point.
Let $\{w_{nk}\}$ be a double sequence of complex numbers. Then there exists an entire function $f$ such that
$f^{(n)}(z_k)= w_{nk}$ for all $k, n$.
\end{fact}

Recall that a commutative domain $B$ is said to be \emph{B\'ezout}, if each 2-generated ideal of $B$ is
principal. This amounts to the so-called \emph{B\'ezout identities}: for each $0\neq a, b\in B$ there are
$c, r, s, u, v$ such that $c= ar+ bs$ and $a= cu$, $b= cv$, hence $c$ generates the ideal $aB+ bB$. Then $c$ is
a \emph{greatest common divisor} of $a$ and $b$, written $\gcd(a,b)$, which is defined up to a multiplicative
unit. Similarly, the notion of a \emph{least common multiple}, $\lcm(a,b)$, makes perfect sense, and
(with a suitable choice of units) we obtain the formula $ab= \gcd(a,b)\cdot \lcm(a,b)$.

The following fact goes back to Weierstrass, but was brought into prominence by Helmer \cite{Hel40}. We will
sketch its proof, borrowed from elsewhere.

\begin{fact}\label{bez}
The ring $E$ of entire complex valued functions is a B\'ezout domain.
\end{fact}
\begin{proof}
We look for a greatest common denominator of $f, g\in E$, i.e.\ an element $h\in fE+ gE$ such that $f, g\in hE$.
We may assume that $f, g$ are nonzero and not invertible. It follows easily that $Z(h)= Z(f)\cap Z(g)$, and the
multiplicity of each $z\in Z(h)$ equals the minimum of $\mu_f(z)$ and $\mu_g(z)$. Choose any such $h$. Since it
divides both $f$ and $g$, canceling by $h$, we may assume that $Z(f)\cap Z(g)= \eset$, hence we have to solve
the equation $fu+ gv= 1$.

In fact, it suffices to find $v\in E$ such that $Z(f)\seq Z(1- gv)$ and $\mu_f(z)\leq \mu_{1-gv}(z)$ for each
$z\in Z(f)$, - then $u$ exists by Fact \ref{princ}. For each $z\in Z(f)$ we will specify few values of $v$ and its
derivatives, and then construct $v$ using Fact \ref{wnk}.

Thus choose $z\in Z(f)$ and assume (for simplicity) that $\mu_f(z)= 3$. Using the standard interpretation of multiple
roots in terms of common roots with derivatives, we need to satisfy the following equalities:

$$
(1- gv)(z)= 0, \qquad (1-gv)'(z)= 0\qquad \text{and} \qquad (1-gv)''(z)= 0\,.
$$

The first condition reads $1= g(z) v(z)$. From $f(z)=0$ it follows $g(z)\neq 0$, hence define $v(z)= - 1/g(z)$. To
satisfy the second and the third equations we set $v'(z)= - g'(z)v(z)/g(z)$ and
$v''(z)= (-g''(z) v(z)- 2g'(z)v'(z))/g(z)$.
\end{proof}

We will need one more property of $E$. Recall that elements $a, b$ of a B\'ezout domain $B$ are called
\emph{coprime} if $\gcd(a, b)= 1$, that is, if $aB+ bB= B$ holds. Following \cite[p. 118]{F-S} we say that $B$
is \emph{adequate} if, for all nonzero noninvertible $a, b\in B$, there is a factorization $a= cd$ such that
$\gcd(c,b)= 1$ and, for each noninvertible divisor $d'$ of $d$, the elements $d'$ and $b$ are not coprime.

In the ring of entire functions the latter means that $Z(d)\seq Z(b)$.

\begin{fact}\label{adeq}
$E$ is an adequate B\'ezout domain.
\end{fact}
\begin{proof}
Let $f, g\in E$ be nonzero and not invertible. Choose $h\in E$ such that $Z(h)= Z(f)\sm Z(g)$, and
$\mu_h(z)= \mu_f(z)$ for each $z$ in this set, in particular $h$ and $g$ are coprime. Then $f= hu$, where
$Z(u)= Z(f)\cap Z(g)\seq Z(g)$, as desired.
\end{proof}

$$
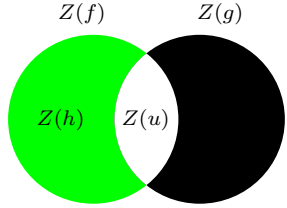
\begin{tikzpicture}[scale=.7]
\draw node at (-1.2,2) {$_{Z(f)}$} node at (1.4,2) {$_{Z(g)}$};
\begin{scope}[fill= green]
\fill[clip] (-1,0) circle (16mm);
\fill[white] (1,0) circle (16mm);
\draw node at (-1.6,0) {$_{Z(h)}$};
\end{scope}
\fill[clip] (1,0) circle (16mm);
\fill[white] (-1,0) circle (16mm);
\draw node at (0,0) {$_{Z(u)}$} node at (-6.4,0) {$_{Z(h)}$};
\end{tikzpicture}
$$

\vspace{3mm}

A (commutative) domain $V$ is said to be a \emph{valuation domain}, if its ideals are linearly ordered by inclusion.
More generally, a domain $R$ is said to be a \emph{Pr\"ufer domain} if, for each prime ideal $P$, the localization
$R_P$ is a valuation domain, - see \cite[Ch. 3]{F-S} for equivalent definitions and properties. Since each B\'ezout
domain $B$ is a Pr\"ufer domain, it follows that any prime ideals $P_1, P_2$ of $B$ included in a maximal ideal
$M$ are comparable, i.e.\ there is no following inclusion diagram for prime ideals.

$$
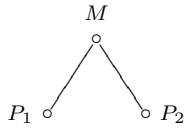
\vcenter{%
\def\objectstyle{\scriptstyle}
\def\labelstyle{\scriptstyle}
\xymatrix@C=10pt@R=20pt{%
&*+={\circ}\ar@{-}[rd]\ar@{-}[ld]\ar@{}+<0pt,10pt>*={M}&\\
*+={\circ}\ar@{}+<-10pt,0pt>*={P_1}&&*+={\circ}\ar@{}+<10pt,0pt>*={P_2}
}}
$$

\vspace{3mm}

For adequate B\'ezout domains no opposite inclusion diagram occurs.

\begin{fact}\label{adeq-fork}(see \cite[Thm. 4]{Hel43})
Let $B$ be an adequate B\'ezout domain. Then every nonzero prime ideal $P$ is contained in a unique maximal ideal,
in particular $B/P$ is a valuation domain.
\end{fact}
Note that the latter statement follows from the former, because each local B\'ezout domain is a valuation domain. For more details on the proof, see a similar situation in Lemma \ref{weak} below (just replace $I$ and $I^{\sh}$ by $P$).

The ring $E$ possesses more remarkable properties, for instance, being adequate, it has \emph{elementary divisors}
and (see \cite{Jen82}) \emph{stable rank 1}, but we will not use these properties in
the paper.

\section{Ideals of B\'ezout domains}\label{S-ideal}

First let us make a trivial remark concerning arbitrary ideals of B\'ezout domains.

\begin{remark}\label{triv}
Let $I$ be an ideal of a B\'ezout domain $B$ and $0\neq a\in I$. Then $b\in I$ if and only if $\gcd(a,b)\in I$.
\end{remark}

Thus to describe $I$ it suffices to look at the divisors $b$ of $a$. For instance, if $B= E$, then the latter
implies that $Z(b)\seq Z(a)$.

We say that a proper ideal $I$ of a B\'ezout domain $B$ is \emph{weakly prime}, if its complement $I^*= B\sm I$ is
closed with respect to least common multiples, i.e.\ $a, b\in I^*$ yields $\lcm(a,b)\in I^*$. Clearly each prime
ideal is weakly prime. On the other hand, for instance, the ideal $z^2 E$ of $E$ is weakly prime but not prime.
These ideals appeared very naturally in \cite{P-T15} and have many nice properties to justify their name. We
mention just a few.

Here is Matlis' like definition, - see \cite{Mat}. Let $I^{\sh}$ consist of elements $r\in B$ such that $ar\in I$
for some $a\in I^*$. For instance $0\in I^{\sh}$, $1\notin I^{\sh}$ and $I\seq I^{\sh}$.

\begin{lemma}\label{sharp}
If $I$ is a weakly prime ideal of a B\'ezout domain $B$, then $I^{\sh}$ is a prime ideal containing $I$. Further
if $P$ is prime ideal, then $P= P^{\sh}$.
\end{lemma}
\begin{proof}
Clearly $I^{\sh}$ is closed with respect to multiplication by elements of $B$. To check that it is closed with
respect to addition, suppose that $r_1, r_2\in I^{\sh}$, hence $a_i r_i\in I$ for some $a_i\in I^*$. Since $I$ is
weakly prime we conclude that $a= \lcm(a_1, a_2)\in I^*$. Then $ar_i\in I$ for every $i$ yields $a(r_1+ r_2)\in I$,
therefore $r_1+ r_2\in I^{\sh}$.

If $P$ is prime, then the inclusion $P^{\sh}\seq P$ follows from the definition of $P^{\sh}$.
\end{proof}

The following result extends Fact \ref{adeq-fork}, with almost the same proof.

\begin{lemma}\label{weak}
Each nonzero weakly prime ideal $I$ of an adequate B\'ezout domain $B$ is contained in a unique maximal ideal, in
particular $B/I$ is a valuation ring, possibly with zero divisors.
\end{lemma}
\begin{proof}
Suppose that $I$ is contained in different maximal ideals $M_1, M_2$, hence $1= q_1+ q_2$ for some
$q_1\in M_1\sm M_2$ and $q_2\in M_2\sm M_1$.

$$
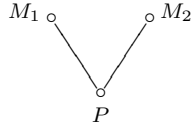
\vcenter{%
\def\objectstyle{\scriptstyle}
\def\labelstyle{\scriptstyle}
\xymatrix@C=10pt@R=20pt{%
*+={\circ}\ar@{}+<-10pt,2pt>*={M_1}&&*+={\circ}\ar@{}+<10pt,2pt>*={M_2}\\
&*+={\circ}\ar@{-}[ru]\ar@{-}[lu]\ar@{}+<0pt,-8pt>*={P}&
}}
$$

\vspace{4mm}

Choose a nonzero $p \in I$. Applying the definition of being adequate to $p$ and $q_1$ we get a factorization
$p= r_1 s_1$, where $\gcd(r_1,q_1)= 1$, and $q_1$ is not coprime to any nonunit dividing $s_1$. From
$\gcd(r_1, q_1)= 1$ and  $q_1\in M_1$ it follows that $r_1\notin M_1$, in particular $r_1\notin I^{\sh}$ . Since $I^{\sh}$ is prime, we derive $s_1\in I^{\sh}$. Similarly $p= r_2 s_2$, where $r_2\notin I^{\sh}$, $s_2\in I^{\sh}$, and $q_2$ is not coprime to any nonunit dividing $s_2$.

From $s_1, s_2\in I^{\sh}$ we conclude that $s= \gcd(s_1, s_2)\in I^{\sh}$. Applying the above condition to $s$ and $q_1$,
and then involving $q_2$, we construct a nonunit dividing both $q_1$ and $q_2$, a clear contradiction.
\end{proof}

The description of maximal ideals of $E$ is well known, and there is a reasonably good (see some comments below)
description of prime ideals of $E$. We approach this classification backwards, first describing weakly prime ideals.
Because it involves ultrafilters on countable sets, we will introduce this terminology.

\subsection{Ultrafilters}\label{ultra}

Let $D$ be a nonempty at most countable set (mostly a subset of $\C$). Recall that a nonempty collection of
subsets of $D$ is said to be a \emph{filter}, if 1) $\eset\notin U$; 2) $U$ is \emph{upward closed}, i.e.\ if
$K\seq L\seq D$, then $K\in U$ implies $L\in U$; 3) $U$ is closed with respect to finite intersections.

The set of filters on $D$ is partially ordered by inclusion, and maximal elements of this ordering are called
\emph{ultrafilters}. In fact $U$ is an ultrafilter iff for any partition $D= K_1\cup K_2$ either $K_1\in U$ or
$K_2\in U$ holds.

For instance, for each $d\in D$, there exists a \emph{principal ultrafilter} $U_d$, namely $K\in U_d$ iff
$d\in K$. If $D$ is finite, then each ultrafilter on $D$ is of this form. Otherwise $U$ is \emph{non-principal},
in particular each cofinite set belongs to $U$, therefore $U$ is not closed with respect to countable
intersections. Further, by \cite[pp. 255--256]{Jech}, there are $2^{\, \bf c}$ ultrafilters on a countable (infinite) $D$.

Let $A$ be any algebraic system in a countable language and let $U$ be an ultrafilter on $D$. Then the elements
of \emph{ultraproduct} $A_U= A^D/U$ are equivalence classes of functions $\mu: D\to A$. Here the functions $\mu$
and $\mu'$ are \emph{equivalent} if they take the same values on a \emph{large} subset of $D$, i.e.\ if the
equalizer $\{d\in D\mid \mu(d)= \mu'(d)\}$ is in $U$.

All operations and relations are naturally transferred to $A_U$, and $A$ is embedded into $A_U$ diagonally. From
\cite[Cor. 4.1.13]{C-K} it follows that $A_U$ is an elementary extension of $A$, in particular $A$ and $A_U$
are elementary equivalent. If $U= U_d$ is principal, then the evaluation map $f\mapsto f(d)$ defines an isomorphism
from $A_U$ onto $A$. Otherwise these algebraic systems are not isomorphic, and (see \cite[Thm. 6.1.1]{C-K}) $A_U$
is \emph{$\om_1$-saturated} of cardinality $\bf c$.
Further, if we assume the Continuum Hypothesis, $\CH$, then
(see \cite[Thm. 6.1.1]{C-K}) the isomorphism type of $A_U$ does not depend on $U$, as soon as $U$ is non-principal.

Our main interest will be when $A= \N$ considered as a linearly ordered abelian semigroup (with respect to addition). If $D$ and $U$ are as above, then let $\mN= \mN_U$ denote the ultraproduct $\N^D/U$. Thus, if $U$ is principal, then $\mN\cong \N$, and otherwise $\mN= \N^{\,D}/U$ is an $\om_1$-saturated linearly ordered abelian semigroup of cardinality $\bf c$.
For instance (taking $D = \omega$ for simplicity), if $U$ is not principal, then the function $\mu(n)= 10n+ 1$ is less than the function $\mu'(n)= n^2$ in
$\mN$, because $10n+1$ is less than $n^2$ for
$n\geq 11$.

As a linear ordering $\mN_U$ contains a least (but no largest) element, and has a lot of simple intervals. Namely,
if $\mu\in \mN_U$, then the function $\mu'(d)= \mu(d)+1$ \emph{covers} $\mu$, i.e.\ $\mu< \mu'$ in $\mN_U$ and there
is no $\mu''$ strictly between $\mu$ and $\mu'$. Further it is easily seen that for $\mu\leq \mu'\in \mN_U$, the
interval $[\mu, \mu']$ is of \emph{finite length} iff the difference $\mu'-\mu$ is \emph{bounded} by some $k$, i.e.
the set $\{d \in D \mid \mu'(d) - \mu(d) \leq k\}$ is large.

We define the functions $\mu, \mu'\in \mN$ to be \emph{finite equivalent}, written $\mu\sim_{fin} \mu'$, if the
interval between $\mu$ and $\mu'$ (or vice versa) is of finite length. Each equivalence class of
$\sim_{fin}$
in $\mN$ is countable. Because $\mN$ is $\om_1$-saturated, it follows that the factor set $\mN'= \mN/\sim_f$ is
a linear ordering of cardinality $\bf c$ which is \emph{dense}, i.e.\ for each $a< b$ in this chain there exists
$c$ such that $a< c< b$.

\subsection{Weakly prime ideals}

In what follows we will use the approach from Gillman--Jerison book \cite{G-J}.

Let $I$ be a nonzero ideal of $E$. Choose $0\neq f\in I$, hence the zero set $D= Z(f)$ is at most countable
and nowhere dense. Let $U_f$ consist of subsets of $D$ of the form $Z(g)$, where $g\in I$. Using the divisibility
properties of entire functions it is easily checked that $U_f$ is a filter on $D$. Then we obtain the following
dichotomy. If there is a $g\in I$ with the smallest $Z(g)\in U_f$, then $I$ is called \emph{fixed}, otherwise $I$
is said to be \emph{free}.

\begin{lemma}\label{max-u}
Let $I$ be a nonzero weakly prime ideal of $E$, $0\neq f\in I$ and $D= Z(f)$. Then $U_f$ is an ultrafilter on $D$.
\end{lemma}
\begin{proof}
We have already mentioned that $U= U_f$ is a filter. To prove that $U$ is maximal, consider a nontrivial partition
$D= K_1\cup K_2$. Let $f_1\in E$ have $K_1$ as its zero set, and multiplicity of each $z\in K_1$ is the same as
for $f$; and similarly define $f_2$. If $f_1, f_2\notin I$ then, by the assumption, $g= \lcm(f_1, f_2)\notin I$.
But $g$ generates the same ideal as $f$, a contradiction.
\end{proof}

Note that, if $f, g\in I$, then $h= \gcd(f,g)\in I$, and $U_h$ is a common restriction of $U_f$ and $U_g$ on the
zero set $Z(h)$.

For fixed weakly prime ideals $(z-t)^k E$, $k\geq 1$ the smallest zero set is the singleton $\{t\}$. Thus to
distinguish weakly prime ideals we need more invariants.

Recall that a \emph{cut} on a linearly ordered set $L$ is a proper partition $L= L_1\cup L_2$ such that $L_2$
is upward closed, hence $L_1$ is downward closed. Clearly each cut is uniquely determined by $L_2$ and vice versa,
hence we will often identify the cut with its upper part.

The following proposition describes weakly prime ideals using cuts on some chains.

\begin{prop}\label{ci}
Let $I$ be a nonzero weakly prime ideal of $E$.

1) If $I$ is fixed then $I= (z-t)^k E$ for some $t\in \C$ and $k\geq 1$.

2) Suppose that $I$ is free. Choose $0\neq f\in I$ and let $D= Z(f)$. Let $c(I)$ consist of multiplicity
functions $\mu_g$, $g\in I$ restricted to $D$, considered as elements of $\mN_{U_f} = \N^D/U_f$. Then $c(I)$ is a cut
on this chain, further $U_f$ and $c(I)$ determine $I$ uniquely.
\end{prop}
\begin{proof}
1) If $I$ is fixed, then choose $0\neq f\in I$ with the least zero set. Since $U_f$ is an ultrafilter, we conclude
that $Z(f)$ is a singleton $\{t\}$. It follows that $I= (z-t)^k E$ for some $k\geq 1$.

2) Suppose that $I$ is free. First we will show that $c(I)$ is upward closed. Suppose that $\mu_g\leq \mu$ modulo
$U$ for some function $\mu\in \N^D$, hence $\mu_g(d)\leq \mu(d)$ for each $d$ in a large subset $K$ of $D$.
We need to construct an entire $u$ such that the restriction of its multiplicity function to $D$ equals $\mu$
modulo $U$.

By the definition of $U_f$
we find $h\in I$ such that $Z(h)= K$. Replacing $h$ by $\gcd(g,h)$ we may assume that $\mu_h(d)\leq \mu(d)$ for
each $d\in K$. Now construct an entire $u$ such that $\mu_u$ restricted to $K$ coincides with $\mu$, and equals
zero otherwise. Then $h$ divides $u$, hence $u\in I$.

It remains to check that $U_f$ and $c(I)$ determine $I$ uniquely. Suppose that $I'\neq I$ is another weakly
prime ideal which contains $f$ and define the same ultrafilter $U_f$ on $D= Z(f)$, and the same cut $c(I)$.
By symmetry we may assume that there exists $g\in I\sm I'$. By the assumption, there exists $g'\in I'$ such
that the restrictions of $\mu_g$ and $\mu_{g'}$ to $D$ equal modulo $U$. Choose a large $K\seq D$ on which these
multiplicity functions coincide. Construct $h\in I$, $h'\in I'$ whose zero sets equal $K$, and
$\mu_h(z)= \mu_{h'}(z)$ for each $z\in K$. It clearly follows that $h\in I\sm I'$, but $hE= h'E$, a contradiction.
\end{proof}

Thus nonzero weakly prime ideals $I, I'$ of $E$ coincide iff for some (or any) $0\neq f\in I\cap I'$ they define
the same ultrafilter $U= U_f$ on the zero set $D= Z(f)$, and the same cut on the corresponding chain $\mN_U$.

The following remark is obvious.

\begin{remark}\label{rem}
Let $I\seq I'$ be nonzero weakly prime ideals of $E$. If \ $0\neq f\in I$ then they define the same ultrafilter
$U=U_f$ on $D= Z(f)$, and $c(I)\leq c(I')$ for corresponding cuts on $\mN_U$, i.e. the upper part of $c(I)$ is
contained in the upper part of $c(I')$.
\end{remark}
\begin{proof}
Clearly $U(I)\seq U(I')$, hence the equality follows from the maximality of ultrafilters. The remaining part is
straightforward.
\end{proof}

Because each prime ideal is completely prime, we recover a well known description of prime ideals of $E$.
Namely, prime ideals $P$ are distinguished by the property that the cut $c(P)$ on the chain $\mN_U$ is
\emph{prime}, i.e., if the equivalence class of a multiplicity function $\mu^k$ is in $c(P)$ for some $k$, then
the same holds true for $\mu$. For instance (taking again $D = \omega$ for simplicity), if $\mu(n)= 2n$ is in $c(I)$, then $\mu'(n)= n$ belongs to there,
but also $\mu''(n)= \lf n/2\rf$.

In particular fixed prime ideals are exactly the maximal ideals $M_t= (z-t)E$, $t\in \C$. If $P$ is not fixed, then,
because all calculations are made modulo a nonprincipal ultrafilter $U$, the property of being prime is quite
tricky. For instance (see \cite{G-H}) for each pair of prime ideals $P\subset P'$ there exist at least
$2^{\aleph_1}$ ideals strictly between $P$ and $P'$. The main idea is that this interval contains a Dedekind
complete $\eta_1$-set of prime ideals, hence \cite[Cor. 13.24]{G-J} gives the desired cardinality.

If we assume the continuum hypothesis, then $\aleph_1= \bf c$, hence the length of a maximal chain of prime
ideals in $E$ equals $2^{\,\bf c}$. However, if we accept the Martin axiom with the negation of $\CH$, then we
see only (following \cite{Jen85}) that this length is at least $2^{\,\aleph_1}= \bf c$. We do not know what is the face value of the Krull dimension of $E$.

Finally we obtain a classical description of maximal ideals $M$ of $E$. Here the corresponding cut $c(I)$
contains all positive multiplicity functions, hence is uniquely determined by the ultrafilter $U$. Thus either
$M$ is fixed, hence equals $M_t= (z-t) E$ for some $t\in \C$; or $M$ is free, therefore is uniquely determined by
the ultrafilter $U_f$ on $D= Z(f)$ for any $0\neq f\in M$. From this it is obvious that each weakly prime ideal
of $E$ is contained in a unique maximal ideal.

\section{Model theory of modules}\label{S-model}

In this section we will recall main notions of the model theory of modules, - for which we refer to \cite{Preb2};
the particular case of B\'ezout domains is treated in detail in \cite{P-T15}.

Let $R$ be a commutative ring. A \emph{positive-primitive formula} $\phi(x)$ in one free variable $x$ is an
existential formula $\ex\, \ov y\, (\ov y A= x \bar b)$, where $\ov y= (y_1, \dots, y_k)$ is a tuple of bound
variables, $A$ is a $k\times l$ matrix over $R$, and $\bar b$ is a row of length $l$. For instance, for each
$a\in R$, we have the \emph{divisibility formula} $a\mid x$ of the form $\ex\, y\, (ya=x)$, and the annihilator
formula $xa= 0$.

Let $N$ be a right $R$-module and choose $m\in N$. We say that $m$ \emph{satisfies} $\phi$ in $N$, written
$N\ms \phi(m)$, if there exists a tuple $\ov m= (m_1, \dots, m_k)$ in $N$ such that $\ov m A= m\ov b$ holds.
For instance $N\ms (a\mid x)(m)$ iff $m= m'a$ for some $m'\in N$, i.e.\ if $m$ is divisible by $a$ in $N$.
Further $N\ms (xa=0)(m)$ iff $ma= 0$.

The corresponding \emph{definable subgroup}, $\phi(N)$, consists of $m\in N$ which satisfy $\phi$. Since $R$ is
commutative, $\phi(N)$ is a submodule of $N$. For instance $(a\mid x)(N)= Na$, and $(xa=0)(N)$ consists of
elements of $N$ which are annihilated by $a$.

We need the following 'elimination of quantifiers' result for pp-formulae over B\'ezout domains.

\begin{fact}\label{elim}(see \cite[L. 2.3]{P-T15})
Let $B$ be a B\'ezout domain. Then each pp-formula $\phi(x)$ is equivalent in the theory of $B$-modules to a finite
sum of formulas $a\mid x\wg xb= 0$, $a, b\in B$; and to a finite conjunction of formulas $c\mid x+ xd=0$,
$c, d\in B$.
\end{fact}

If $B= E$ then a further reduction is possible. For instance, if $0\neq c$ is not a unit and $d$ is nonzero, then
one may assume that $Z(d)\seq Z(c)$ in the latter formula. Namely decompose $d= ed'$ such that $\gcd(e,c)= 1$
according to the definition of being adequate, in particular $Z(d')\seq Z(c)$. Since $c$ and $e$ are coprime, it
follows from \cite[Sect. 3]{P-T15} that the formulae $c\mid x+ xd=0$ and $c\mid x+ xd'=0$ are equivalent. Further,
using elementary duality, we may also assume that, if $0\neq b$ is not a unit and $a\neq 0$, then $Z(a)\seq Z(b)$
in the former formula.

An inclusion of modules $N\seq N'$ is said to be \emph{pure} if, for each $m\in N$ and each pp-formula $\phi$,
from $N'\ms \phi(m)$ it follows that $N\ms \phi(m)$. For instance each injective module is pure in any its
overmodule. We say that a module is \emph{pure injective} if it is injective with respect to pure embeddings.
For example, each injective module is pure injective, and the same holds true for each $R$-module of finite length.

The isomorphism types of indecomposable pure injective modules form points of a topological space, the
\emph{Ziegler spectrum} of $R$, $\Zg_R$. In fact there at most $2^{\,\max(|R|, \aleph_0)}$ such points. The
topology on this space is given by (quasi-compact) basic open sets $(\phi/\psi)$, where $\phi, \psi$ range over
pp-formulae in one variable. Here $(\phi/\psi)$ consists of points $N$ in $\Zg_R$ such that $\phi(N)$ is not a
subset of $\psi(N)$. We will often refer as '$\phi$ over $\psi$' to this set. For instance, the open set $xa=0$
over $x=0$ consists of indecomposable pure injective modules containing a nonzero element annihilated by $a$.

For B\'ezout domains Fact \ref{elim} provides a better basis for Ziegler topology.

\begin{fact}\label{z-basis}
Let $B$ be a B\'ezout domain. Then the basic open sets $a\mid x\wg xb=0$ over $c\mid x+ xd=0$, $a, b, c, d\in B$
form an open basis of Ziegler topology.
\end{fact}

Of course some such pairs of pp-formulae define empty sets, hence redundant. A precise criterion when this happens
can be extracted from \cite[Sect. 4]{P-T15}.

Since $E$ is a Pr\"ufer domain, each indecomposable pure injective module $N$ is \emph{pp-uniserial}, i.e.\ the
lattice of definable subgroups of $N$ is a chain. It follows that each basic open set as above equals to the
intersection of the following open sets: 1) $a\mid x$ over $xb=0$; 2) $a\mid x$ over $c\mid x$; 3) $xb=0$ over
$c\mid x$, and 4) $xb=0$ over $xd= 0$, hence these sets give a \emph{subbasis} for the Ziegler topology.

The support of some such pairs is easily understood. For instance, look at the pair $a\mid x$ over $c\mid x$.
If it is nontrivial then, taking the conjunction and using \cite[L. 3.1]{P-T15}, we may assume that $a\neq 0$ and
$c= ga$ for some nonunit $g$. This pair opens a point $N$ iff $Nga$ is a proper subset of $Na$. Since $N$ is
pp-uniserial, this is the same as $Ng\subset N$ and $Na\neq 0$. Thus we can further decompose this basic open
set into the intersection of open sets $x=x$ over $g\mid x$, and $x=x$ over $xa=0$.

However, we see no real advantage in working with this subbasis, because (we thank Lorna Gregory for this remark)
the intersection of arbitrary such pairs, say $xb=0$ over $x=0$, and $x=x$ over $a\mid x$ may be non-compact,
hence equals to an infinite union of basic open sets.

Let $N$ be an $R$-module and let $m$ be a nonzero element of $N$. The \emph{positive primitive type} of $m$ in $N$,
written $pp_M(m)$, consists of pp-formulae $\phi$ such that $m$ satisfies $\phi$ in $N$, in particular this set
is closed with respect to finite conjunctions and implications. The converse is also true: if $p$ is a collection
of pp-formulae closed with respect to finite conjunctions and implications, then there exists a module $N$ and
its element $m$ such that $p= pp_N(m)$.

A pp-type $p$ is said to be \emph{indecomposable} if it is realized by a nonzero element in an indecomposable
pure injective module. This module is unique up to an isomorphism over the realization, and is called the
\emph{pure injective envelope of $p$}, written $\PE(p)$. Note that different pp-types may lead to isomorphic
pure injective envelopes, for example, this is the case when $N$ is an indecomposable pure injective module and
$p= pp_M(m)$, $q= pp_M(mr)$, where $mr$ is nonzero, thus $q$ is a \emph{direct shift} of $p$.

Now we specialize to B\'ezout domains. First we refine the classification of indecomposable pp-types from
\cite[Thm. 4.5]{P-T15}.

\begin{lemma}\label{t-class}
Let $B$ be a B\'ezout domain. Then there exists a natural one-to-one correspondence between indecomposable
pp-types $p$ in one variable in the theory of $B$-modules and the pairs $(I,J)$ such that the following holds.

1) The \emph{annihilator ideal} $I= I(p)$, consisting of $b\in B$ such that $xb=0\in p$, is a weakly prime ideal.

2) The \emph{non-divisibility} ideal $J= J(p)$, consisting of $a\in B$ such that $a\mid x$ is not in $p$, is
a weakly prime ideal.

3) $I^{\sh}$ and $J^{\sh}$ are comparable prime ideals.
\end{lemma}

Such pairs are called \emph{admissible} in \cite{P-T15}.

\begin{proof}
The only difference with \cite{P-T15} is that in there 3) is formulated as follows. If $d\in I^*$ divides
$b\in I$, and $c\in J^*$ divides $a\in J$, then the quotients $b/d$ and $a/c$ are not coprime. This means that
$b/d\in I^{\sh}$ and $a/c\in J^{\sh}$, hence that $I^{\sh}+ J^{\sh}$ is a proper ideal. The remaining part is
straightforward.
\end{proof}

Note that, if $I=0$ and $J=0$, then 3) gets trivial. On the other hand, if $B= E$ and $I, J$ are nonzero, then 3)
means that there is $0\neq f\in I\cap J$, and the ultrafilters $U_f$ on $D= Z(f)$ defined by $I$ and $J$
coincide, and there are no further restrictions. Thus we are led to the following definition.

A triple $(U, I, J)$ is said to be \emph{admissible}, if $I, J$ are weakly prime ideals of $E$ such that one
of the following holds.

1) $I= J= 0$ and $U$ is an empty.

2) $I$ is nonzero, $J= 0$ and, for some $0\neq f\in I$, $U= U_f$ is an ultrafilter on $D= Z(f)$ corresponding
to $I$.

3) $I=0$, $J$ is nonzero and, for some $0\neq g\in J$,  $U= U_g$ is an ultrafilter on $D= Z(g)$ corresponding
to $J$.

4) $I, J\neq 0$ and there is $0\neq h\in I\cap J$ such that $U= U_h$ is an ultrafilter on $D= Z(h)$ defined
by both $I$ and $J$.

When $I$ or $J$ are nonzero, they define the cuts $c(I)$ and $c(J)$ on the ultraproduct $\mN_U= \N^D/U$,
and are uniquely determined by these cuts. We will often identify ideals with the corresponding cuts.

Note that the triples $(U, I, J)$ and $(U',I',J')$ in 4) produce the same pp-type iff $I= I'$, $J= J'$, hence $U$
and $U'$ have a common restriction to $Z(f)\cap Z(g)$, and similarly for 2) and 3). For instance, if $U$ is
defined on some $D$ and generated by $t\in D$, then $I= (z-t)^k E$ and $J= (z-t)^l E$, and these $k, l\geq 1$
uniquely determine the pp-type.

In particular there is a unique pp-type corresponding to the pair $I= J= 0$ as in 1). This pp-type is realized by
any nonzero element in the quotient field $Q$ of $E$, which is the field of meromorphic functions.

We will denote by $p(U,I,J)$ the indecomposable pp-type associated to an admissible triple $(U,I,J)$, and by
$\PE(U,I,J)$ the corresponding indecomposable pure injective module.

It follows from \cite[Thm. 5.4]{Zie} that over a commutative ring $R$ each indecomposable pure injective module
$N$ localizes. Namely define the \emph{localizing ideal} $P= P(N)$ to consist of elements of $R$ which do not act by
multiplication as automorphisms of $N$. Then $P$ is a prime ideal and $N$ is (pure injective indecomposable)
module over the localization $R_P$. This ideal is easily recognized in our setting.

\begin{lemma}\label{types}
Let $(U,I,J)$ be an admissible triple over $E$, and let $N= \PE(U,I,J)$ be the corresponding indecomposable
pure injective module. Then the localizing ideal of $N$ is the prime ideal $I^{\sh}\cup J^{\sh}$.
\end{lemma}
\begin{proof}
Choose $m\in N$ which realizes $p= p(U,I,J)$. If $f\in E$ then it is easily checked that the multiplication by $f$
does not increase $p$ iff $f\notin I^{\sh}\cup J^{\sh}$, from which the result follows.

Namely, for every $f\in I^{\sh}$ there is $g\in I^*$ such that $gf\in I$. It follows that $mg\neq 0$ and
$mgf= 0$, hence $f$ cannot determine an automorphism of $N$; and similarly if $f\in J^{\sh}$. Conversely, let
$f\notin I^{\sh}\cup J^{\sh}$. Then $I^{\sh}$ and $J^{\sh}$ are preserved  under multiplication by $f$.
Thus this multiplication does not increase $p$ and determines an automorphism of $N$.
\end{proof}

Another possibility to grasp the meaning of this ideal is the following. We have $f\notin I^{\sh}\cup J^{\sh}$
iff $Z(f)$ is separated from $U$, i.e.\ if there exists $K\in U$ such that $Z(f)\cap K= \eset$.

Having described indecomposable pp-types, we wish to classify their envelopes, i.e.\ indecomposable pure injective modules. To determine points of $\Zg_E$, it remains to describe the equivalence relation on such
pp-types which correspond to the isomorphism relation on their envelopes. We have already mentioned the
typical occurrence of such identification: the shift by an element of the ring.

It follows from \cite[L. 4.7]{P-T15} that for B\'ezout domains this is the only possibility: if $m, m'$ are
nonzero elements in an indecomposable pure injective module $N$, then there exists $r\in B$ such that
$pp_N(mr)= pp_N(m')$ or $pp_N(m)= pp_N(m'r)$, hence these types are identified by either direct or inverse
shift. This leads to a simple description of this equivalence relation on the level of admissible pairs.
We say that admissible pairs $(I,J)$ and $(I',J')$ are \emph{equivalent}, if their pure injective envelopes
are isomorphic. By \cite[L. 4.6, 4.7]{P-T15} this happens iff one of the following holds.

1) There exists $a\notin I$ such that $I'= (I:a)= \{b\in B\mid ab\in I\}$ and
$(J')^*= J_a^*= \{b\in B\mid b/\gcd(a,b)\in J^*\}$, the direct shift by $a$,

2) The symmetric condition with $(I,J)$ and $(I',J')$ interchanged, the inverse shift by $a$.

Note that the direct or inverse shift of the zero ideal is zero again, furthermore such shifts do not change
prime ideals.

For $E$ the above shifts correspond to a simultaneous shifting of the pair of cuts. Namely, we choose a function
$\mu$ in the lower part of $c(I)$, subtract it from the multiplicity function of each $f\in I$ to get $I'$, and
add this function to the multiplicity function of each $g\in J$ to get $J'$; or make a similar construction
starting with $\mu$ in the lower part of $c(J)$.

For instance, suppose that $I$ is a principal cut generated by the function $\mu(n)= n+1$ in the ultraproduct
$\N^D/U$ for some zero set $D$ identified with $\om$; and let $J$ correspond to the the principal cut on $D$
generated by $\eta(n)= n^2$. Then the function $\rho(n)= n$ is in the lower part of $c(I)$. Taking the direct
shift by $\rho$, we obtain an equivalent pair $(I',J')$, where $I'$ is the maximal ideal $M_U$, such that
$c(I')$ is generated by $\mu'(n)=1$; and $c(J')$ is generated by $\eta'(n)= n^2+n$.

Thus we have obtained the following description of points of $\Zg_E$.

\begin{theorem}\label{points}
Let $E$ be the ring of entire functions. There is a natural one-to-one correspondence between points of the
Ziegler spectrum of $E$ (hence isomorphism type of indecomposable pure injective modules) and admissible triples
$(U,I,J)$ with respect to the following equivalence relation.

1) For nonzero $I,J, I',J'$, the triples $(U,I,J)$ and $(U',I',J')$ are equivalent iff $U$ and $U'$ can be
restricted on a common zero set $D\in U, U'$ such that the restriction of cuts corresponding to $(I,J)$
and $(I',J')$ can be identified by a shift.

2) If $I=0$ but $J$ is nonzero, then $(U,I,J)$ and $(U',I',J')$ are equivalent iff $I'=0$ and $J, J'$ can
be restricted to a common zero set $D\in U, U'$ such that the restriction of cuts corresponding to $J$ and
$J'$ on $D$ can be identified by a shift.

3) If $I$ is nonzero but $J=0$, then $(U,I,J)$ and $(U',I',J')$ are equivalent iff $J'=0$ and $I, I'$ can
be restricted to a common zero set $D\in U, U'$ such that the restriction of cuts corresponding to $I$ and
$I'$ on $D$ can be identified by a shift.

4) If $I= J= 0$, then we have only one admissible triple $(\eset,0,0)$ in this equivalence class.
\end{theorem}

\section{The Ziegler spectrum}\label{S-zieg}

In the previous section we have described the points of the topological space $\Zg_E$. In this section we will
touch upon the topology. First we estimate the number of points in this space.

\begin{prop}\label{numb}
The cardinality of the Ziegler spectrum of $E$ equals $2^{\,\bf c}$.
\end{prop}
\begin{proof}
Since the cardinality of $E$ is continuum, we conclude that $|\Zg_E|\leq 2^{\,\bf c}$. On the other hand, chosen
a nonzero countable subset $D$ of $\C$, one can construct $2^{\,\bf c}$ ultrafilters $U$ on $D$, hence the same amount
of free maximal ideals of $E$. When $M$ ranges over these maximal ideals, then the admissible triples $(U,M,0)$
provide non-isomorphic indecomposable pure injective modules.
Namely, if $a\notin M$ then $(M:a)= M$, hence the direct or inverse shift does not change the corresponding cut.
\end{proof}

In fact the above constructed points can be separated from each other using Ziegler topology. Namely, assume that
$U, U'$ are different ultrafilters on $D$, hence there is a zero set $Z(g)\seq D$ which is in $U$ but not in $U'$.
Then $g$ acts with torsion on $\PE(U,M,0)$, but as an automorphism on $\PE(U',M,0)$, hence the former point is
separated from the latter by the pair $xg=0$ over $x= 0$. Thus $\Zg_E$ has a collection of $2^{\, \bf c}$ points
which can be pairwise separated, hence not elementary equivalent.

We will employ the following point of view on the Ziegler spectrum of any B\'ezout domain $B$. Because each point
of $\Zg_B$ localizes, the whole space is covered by the closed subsets, the Ziegler spectra of localizations
$B_P$ for prime (or just maximal) ideals of $B$. If we consider these spaces as 'stalks', then the topology
on $\Zg_B$ is patched from these topologies using basic open sets from Fact \ref{z-basis}.

Each $B_P$ is a valuation domain, and the Ziegler spectrum of this class of rings was thoroughly investigated
(see \cite[Ch. 12, 13]{Punb}, or \cite{Gre15} for recent development). In more detail, let $\Ga$ denote the
value group of a valuation domain $V$. The nonnegative part $\Ga^+$ of $\Ga$ can be identified as a poset with
principal ideals of $V$. We use the first copy of $\Ga^+$ to represent annihilator formulae, and its second copy
to encode divisibility formulae. In this way the sum $a\mid x+ xb=0$ is represented by the point $(b,a)$ on
the quarter plane $\Ga^+\times \Ga^+$, and each pp-formula corresponds to a finite collection (conjunction) of
such points. Further the whole lattice of pp-formulae over $V$ is a free product of these two chains in the
variety of modular lattices, in particular it is distributive.

Also indecomposable pp-types over $V$ correspond to pairs of cuts $(I,J)$ on $\Ga^+$, hence are represented
as points on the completed quarter plane $\wh\Ga\times \wh \Ga$, or rather by lines on this plane of slope $-1$
(moving along the line corresponds to taking shifts). If $a, b, c,d\neq 0$, then the basic open set
$a\mid x\wg xb=0$ over $c\mid x+ xd=0$ in the Ziegler spectrum is interpreted in \cite[Sect. 4]{PPT} as the
rectangle $(d,b]\times (c,a]$ on the plane 'catching' an indecomposable pure injective module, if its line
intersects this rectangle.

To recover topology consider the 'generic' case of a basic open set $(\phi/\psi)$, where
$\phi\doteq a\mid x\wg xb=0$ and $\psi\doteq c\mid x+ xd=0$ for some nonzero noninvertible $a, b, c, d\in B$.
Using a standard trick (see \cite{P-T15}) one may assume that $c= ga$ and $b= dh$ for nonunits $g, h\in B$.

Now suppose that $P$ is a prime ideal of $B$ and $V= B_P$. If $g\notin P$, then the above open set $(\phi/\psi)$
is trivial when restricted to $\Zg_V$; and the same holds true when $h\notin P$. Otherwise $g, h\in P$, and we
will intepret this open set as the above rectangle $(d,b]_P\times [c,a]_P$ (over $V$). Thus the basic open set $(\phi/\psi)$
can be thought of as a sheaf of rectangles when $P$ runs over prime ideals.

We will demonstrate few instances of this approach applied to the Ziegler spectrum of $E$. Recall that the ring
of quotients of $E$ is the field $Q$ of meromorphic functions. Since this module is indecomposable and injective,
it is a point of $\Zg_E$. Further, for each $t\in \C$ and each $k\geq 1$, the module $E_t(k)= E/(z-t)^k E$ is
indecomposable of finite length, hence is also a point in $\Zg_E$.

First we will describe isolated points in $\Zg_E$.

\begin{theorem}\label{isol}
The finite length points $E_t(k)$, $t\in \C$, $k\geq 1$ are isolated and dense in $\Zg_E$. Those are the
only isolated points in this space.
\end{theorem}
\begin{proof}
First we will check that each point $E_t(k)$ is isolated. Namely set $a= 1$, $c= z-t$, $b= (z-t)^k$,
$d=(z-t)^{k-1}$ and consider the basic open set $(\phi/\psi)$, where $\phi\doteq a\mid x\wg xb=0$ and
$\psi\doteq c\mid x+ xd=0$. Clearly this pair opens the module $E_t(k)$ on the element $\bar 1$. Suppose that
this pair opens an indecomposable pure injective module $N$ on an element $m$. If $I$ is the annihilator of $m$,
then $(z-t)^k\in I$ and $(z-t)^{k-1}\notin I$ yields $I= (z-t)^k$. Similarly for the non-divisibility ideal $J$ of
$N$ we obtain $z-t\in J$, hence $J= (z-t)E$. Thus we conclude that $N$ is isomorphic to $E_t(k)$.

Now we would like to show that these points are dense in $\Zg_E$. It suffices to check that each nontrivial basic
open set $(\phi/\psi)$, where $\phi\doteq a\mid x\wg xb=0$ and $\psi\doteq c\mid x+ xd=0$, contains such a point.
We may assume that this open set contains a point not from the list, say a point $\PE(U,I,J)$, where
$I, J$ are nonzero, $0\neq f\in I, J$ and $U= U_f$ is a nonprincipal ultrafilter on $D= Z(f)$. By refining
$D$ we may assume that $\mu_c(z)< \mu_a(z)$ and $\mu_d(z)< \mu_b(z)$ for each $z\in D$, and choose any $t\in D$.
Since all multiplicities are natural numbers, it is easy, for some $k$, to shift a pp-type of $\ov 1$ in $E_t(k)$
in this interval, as desired.

Similar arguments apply when $I=0$ or $J=0$.
\end{proof}

Having described isolated points, we will look at the closed ones. We need the following auxiliary result.

\begin{lemma}\label{endo-fin}
Let $V$ be a valuation domain and let $N$ be an indecomposable finite endolength point  in the Ziegler spectrum
of $V$. Then one of the following holds.

1) $N$ is the quotient field $Q(V/P)$ for some factor $V/P$ by a prime ideal $P$, the \emph{generic} point.

2) $N$ is isomorphic to $E_P(k)= V_P/P_P^k$, $k\geq 2$, where $P$ is a prime ideal of $V$ such that the ideal
$P_P$ is not idempotent.
\end{lemma}

Here we excluded the case $k=1$ in 2), because the factor $V_P/P_P$ is isomorphic to $Q(V/P)$.

\begin{proof}
Each module $Q(V/P)$ has endolength one, and each module $E_P(k)$ has finite length over $V_P$, hence is of finite
endolength over $V$.

Suppose that $N$ is an indecomposable finite endolength $V$-module. It follows that $N$ is $\Sig$-pure injective,
i.e.\ has a d.c.c.\ on definable subgroups. The structure of such modules over valuation domains is well known
(see \cite[Ch. 16]{Punb} for a more general setting). Namely, let $I$ denote the annihilator of $N$ and let
$P$ be the localizing ideal of $N$, hence $I\seq P$.

Then $N$ is a $V_P$-module, furthermore $V'= V_P/I_P$ is a noetherian valuation ring and $N$ is isomorphic
to the injective envelope (over this ring) of the unique simple $V'$-module $V_P/P_P$.

If $V'$ is not artinian, then $N$ has the ascending chain of definable (annihilator) subgroups, hence is not
of finite endolength, a contradiction. Thus $V'$ is artinian, hence self-injective, and $N$ is isomorphic to $V'$,
i.e.\ to $V_P/I\cong V_P/P_P^k$ for some $k\geq 1$.

If $k=1$, then $V'= V_P/P_P\cong Q(V/P)$, hence $N$ is generic. Otherwise we may assume that $P_P$ is not
idempotent.
\end{proof}

Note that this description works equally well for any B\'ezout domain $B$. Because it is difficult to decide
in this general framework when the maximal ideal $P_P$ of the localization $B_P$ is idempotent, we
will prefer to stay down to living examples. For instance, if $B= \A$ is the ring of algebraic integers, then
one could take square roots, hence each prime ideal is idempotent. This is almost the case for $E$ with few
exceptions, - see below.

Note that the lattice of pp-formulae of a B\'ezout domain $B$ is always distributive, hence the same holds true for
any theory $T$ of $E$-modules. It follows from  \cite[Thm. 5.3.28]{Preb2} that the \emph{isolated condition}
holds true: each isolated point in $T$ is isolated by a minimal pair. Now from \cite[Cor. 5.3.23]{Preb2} we
conclude  that a point in the Ziegler spectrum of this theory is closed iff it is of finite endolength.

However one should be cautious when using Lemma \ref{endo-fin} in this general setting - this lemma applies just
to the theory of \emph{all} modules. This is exactly the case we investigate now.

\begin{prop}\label{closed}
The following is a complete list of closed points of $\Zg_E$.

1) The generic modules $Q(E/P)$, where $P$ runs over prime ideals of $E$. In particular, when $P=0$, we obtain
the field $Q$ of meromorphic functions.

2) The modules $E_t(k)= E/(z-t)^k E$, $t\in \C$, $k\geq 2$.

3) The modules $E_M(k)= E/M^k$ for each free maximal ideal $M$ and $k\geq 2$.
\end{prop}
\begin{proof}
Clearly all such points are of finite endolength, hence closed.

Let $N$ be a closed point, and let $P$ be its localization ideal, in particular $N$ is a closed point in
the Ziegler spectrum of the valuation domain $E_P$. Using Lemma \ref{endo-fin} we may assume that
that $N\cong E_P(k)$, $k\geq 2$, where the ideal $P_P$ of $E_P$ is not idempotent.

It is easily seen that, if $P$ is a non-maximal prime ideal of $E$, then there are square roots in $P$.
We conclude that $P$ is idempotent, therefore occurs just in case 1). Thus we may assume that $P=M$ is maximal.

If $M$ is fixed, then $M= (z-t)E$, hence not idempotent. Furthermore clearly the localization $E_M$ is
isomorphic to the power series ring $\C[[X]]$, hence $N$ is isomorphic to $E_t(k)$.

Suppose that $M$ is free and maximal, with corresponding nonprincipal ultrafilter $U$. The intersection
$M_{\fty}$ of powers of $M$ is a prime ideal of $E$ consisting of functions which are not constant modulo $U$,
and this ideal is not idempotent. Furthermore being the intersection of powers $M^k$, this ideal clearly
annihilates $N$. Also it follows from \cite[Thm. 8]{Hen53} that the factor $E/M_{\fty}$ is isomorphic to
$\C[[X]]$. It is easily derived that $N\cong E_M(k)$.
\end{proof}

Dropping from $\Zg_E$ the isolated points we obtain $\Zg'_E$, the first \emph{Cantor--Bendixson derivative} of
this space, with the induced topology. This class of modules generates the theory $T'_E$, the
\emph{$\CB$-derivative} of the theory $T_E$ of all $E$-modules.

\begin{theorem}\label{e-prime}
The theory $T'_E$ coincides with the theory of $E_S$-modules, where $S$ is the multiplicative closed set
consisting of nonzero polynomials.
\end{theorem}
\begin{proof}
Note that for each point $t\in \C$ and each nonprincipal ultrafilter $U$ on a zero set $D$, we have
$D\sm \{t\}\in U$, therefore $z-t$ acts by multiplication as an automorphism on each indecomposable
pure injective module corresponding to $U$.

It follows that each point of $\Zg'_E$ is defined over $E_S$ (we put for simplicity from now on $E' = E_S$). On the other hand, it is not difficult to check
that $E'$ is the model of $T'_E$, hence (see \cite[Cor. 6.1.5]{Preb2}) the ring of definable scalars of
$T'_E$ coincides with $E'$.
\end{proof}

Thus, after taking the first derivative, we obtain a more regular B\'ezout domain $E'$. Further, because of
the isolated  condition, the lattice of pp-formulae of $T'_E$ is obtained from the lattice of pp-formulae of
$E$ by collapsing intervals of finite length.

We will not need higher $\CB$-derivatives, because of the following result.

\begin{theorem}\label{isol-1}
$\Zg'_E$ has no isolated points. Furthermore no nontrivial interval in the lattice of pp-formulae of $T'$
is a chain.
\end{theorem}
\begin{proof}
Since the theory of all $E$-modules enjoys the isolation condition, the latter statement implies the former.

Clearly it suffices to prove the claim for each localization $V= E_M$, where $M$ is a free maximal ideal with corresponding ultrafilter $U$.
Let $L$ be the lattice of pp-formulae of $\Zg_V$. Then $L$ is freely generated by two copies of the chain
$\mN_U$. We put to use results of \cite{Pun99}. Namely, the effect of the first step of the $\CB$-analysis on
the lattice $L$ is that it collapses the intervals of finite length on each of two copies of $\mN$.

Thus the lattice $L'$ is freely generated by two copies of the derivative chain $\mN'$. We have already seen
in Section \ref{ultra} that this chain is dense. It easily follows that no nontrivial interval in $L'$ is a
chain, as desired.
\end{proof}

For a definition of width and breadth of a lattice see \cite[Sec. 7.1]{Preb2}. It follows from
Theorem \ref{isol-1} that both dimensions are undefined for the theory of $E'$-modules, and hence for
$E$-modules. Furthermore in \cite{P-T15} we constructed a superdecomposable pp-type, hence a superdecomposable
pure injective module over $E$. It follows that nonzero polynomials act as automorphisms on this module, hence
it is defined over $E'$. Below we will show that the pure injective envelope of $E'$ itself is superdecomposable
if viewed as a module over $E'$ and consequently over $E$.

But before that let us consider the closed points in $T'_E$. From the above discussion it follows that they are of
finite endolength. Note that the prime ideals of $E_S$ one-to-one correspond to free prime ideals of $E$.

\begin{lemma}\label{close-1}
The closed points in $\Zg'_E$ are exactly the generic points $Q(E/P)$, where $P$ runs over free prime ideals of $E$.
\end{lemma}
\begin{proof}
Following the proof of Proposition \ref{closed}, it suffices to notice that each prime ideal $P$ of $E'$
is idempotent. Namely, the only case we have not considered is when $P$ corresponds to the free maximal ideal
$M$. However, after localizing, we obtain $M_S= (M_{\fty})_S$, hence this ideal is idempotent.
\end{proof}

Recall that a module $M$ is said to be \emph{superdecomposable}, if no nonzero direct summand of $M$ is
indecomposable. We need the following general fact.

\begin{lemma}\label{envelope1}
Let $B$ be a commutative B\'ezout domain. Then the following are equivalent.

1) The pure injective envelope of $B$ as a module over itself is superdecomposable.

2) If $0\neq a\in B$ is not invertible, then there are coprime nonunits $a_1, a_2 \in B$ dividing $a$.
\end{lemma}
\begin{proof}
Since $B$ is coherent, each pp-definable subgroup in $B$ (as a module) is a principal ideal
(see \cite[Thm. 2.3.19]{Preb2}). Let $p= pp_B(1)$ denote the pp-type taken in the theory of $B$ (i.e. in the theory
of flat = torsion free $B$-modules), hence $p$ is a filter in the lattice of principal ideals of $B$.

Then 1) says that $p$ is superdecomposable, i.e. contains no \emph{large formulas}. Since $B$ is distributive, this
is the same as to say that for each $\phi\in p^-$ there are $\phi_1, \phi_2\in p^-$ such that $\phi\to \phi_i$ and
$\phi_1+ \phi_2\in p$, i.e. $\phi_1+ \phi_2$ is a trivial formula. Replacing formulas by ideals they define, we
obtain the desired.
\end{proof}

We apply this criterion to our setting.

\begin{prop}\label{envelope2}
The pure injective envelope of $E'$ (over $E'$ and hence over $E$) is a superdecomposable module.
\end{prop}
\begin{proof}
Suppose that $fE'$ is a proper ideal of $E'$, hence we may assume that $f \in E$,
$f$ is not a polynomial and $Z(f)$ is an infinite countable set. Let $Z(f)= I_1\cup
I_2$ be a partition of $Z(f)$ into infinite sets. Choose $f_1\in E$ such
that $Z(f_1)= I_1$ and $f\in f_1 E$. Then $f= f_1 f_2$, where both functions are noninvertible in $E'$ and
$Z(f_2)= I_2$, hence $f_1$ and $f_2$ are coprime.
\end{proof}

The following question naturally arises from the previous results.

\begin{ques}
Describe all representations of $\PE(E')$ as a pure injective envelope of direct sums of pure injective modules.

For instance, describe all direct summands of this module, and all direct sum decompositions
$\PE(E')= N_1\oplus N_2$.
\end{ques}

\end{document}